\documentclass[pdflatex,sn-mathphys-num]{sn-jnl}

\usepackage{graphicx}%
\usepackage{multirow}%
\usepackage{amsmath,amssymb,amsfonts}%
\usepackage{amsthm}%
\usepackage{mathrsfs}%
\usepackage[title]{appendix}%
\usepackage{xcolor}%
\usepackage{textcomp}%
\usepackage{manyfoot}%
\usepackage{booktabs}%
\usepackage{algorithm}%
\usepackage{algorithmicx}%
\usepackage{algpseudocode}%
\usepackage{listings}%

\theoremstyle{thmstyleone}%
\newtheorem{theorem}{Theorem}
\theoremstyle{thmstyletwo}%
\newtheorem{remark}{Remark}%

\theoremstyle{thmstylethree}%
\newtheorem{assumption}{Assumption}%

\begin{document}

\title[New reformulations for 0-1 quadratic programming problem]{New reformulations for 0-1 quadratic programming problem using quadratic nonconvex reformulation techniques and valid inequalities}


\author[1]{\fnm{Cheng} \sur{Lu}}\email{lucheng1983@163.com}

\author[1]{\fnm{Yu} \sur{Fei}}\email{feiyv2003@163.com}

\author*[2]{\fnm{Jing} \sur{Zhou}}\email{zhoujing@zjut.edu.cn}

\author[3]{\fnm{Zhibin} \sur{Deng}}\email{zhibindeng@ucas.edu.cn}

\author[4]{\fnm{Guangtai} \sur{Qu}}\email{guangtaiqu@163.com}

\affil[1]{\orgdiv{School of Economics and Management}, \orgname{North China Electric Power University}, \city{Beijing}, \postcode{102206}, \country{China}}

\affil[2]{\orgdiv{School of Mathematical Sciences}, \orgname{Zhejiang University of Technology}, \street{Liuhe Road}, \city{Hangzhou}, \postcode{310023}, \state{Zhejiang Province}, \country{China}}

\affil[3]{\orgdiv{School of Economics and Management}, \orgname{University of Chinese Academy of Sciences}, \city{Beijing}, \postcode{100190}, \country{China}}

\affil[4]{\orgdiv{School of Management}, \orgname{Beijing Institute of Technology}, \city{Beijing}, \postcode{100081}, \country{China}}


\abstract{It is well-known that the quadratic convex reformulation (QCR) technique can speed up some general-purpose solvers such as CPLEX and Gurobi. Recently, the method of quadratic nonconvex reformulation (QNR) was proposed, which provides an alternative way for accelerating a solver via reformulation technique. This paper proposes several new reformulations for 0-1 quadratic programming problems using the QNR technique. Such a technique provides more flexibility in adding nonconvex quadratic constraints into the problem formulation, so that some valid inequalities, such as the triangle inequalities, can be incorporated into the formulation to tighten the lower bound of the problem. We analyze the effects of the proposed reformulations on the lower bounds implemented in the solver, and propose some methods to maximize the McCormick relaxation bounds of the reformulations. Our numerical experiments compare the proposed reformulations with the existing quadratic convex reformulations, showing the effectiveness of the proposed reformulations on 0-1 quadratic programming problems.}

\keywords{0-1 quadratic programming, reformulation, triangle inequalities}



\maketitle

\section{Introduction}

Consider the binary quadratic programming problem of the following form:
\begin{equation}\label{BQP}\tag{BQP}
	\begin{aligned}
		\min ~&~ \frac{1}{2}x^\top Q x+ c^\top x \\
		\text{s.t.} ~&~ x_i \in \{0,1\},~i=1,\ldots,n,
	\end{aligned}
\end{equation}
where $x$ is a binary vector, $Q$ is a symmetric matrix of order $n$, and $c$ is a vector of linear coefficients. \eqref{BQP} is a well-known NP-hard problem \cite{Garey}, which has wide applications \cite{BQP}. This paper focuses on exact algorithms for solving \eqref{BQP}.

Several tailored solvers, including Biq Mac \cite{Rendl}, BiqCrunch \cite{Krislock,BiqCrunch}, BiqBin \cite{BiqBin}, and FixingBB \cite{Locatelli2024}, have been developed for finding exact solutions of \eqref{BQP}. In addition to these tailored solvers, some general-purpose solvers, such as Gurobi, CPLEX, Baron \cite{Sahinidis,Tawarmalani}, SCIP \cite{Bestuzheva2025,Vigerske}, can also solve this problem exactly. Currently, there is still a significant gap between the performance of these general-purpose solvers and the tailored solvers on solving \eqref{BQP}, in terms of solving efficiency. The purpose of this paper is to propose some new reformulations for \eqref{BQP} to accelerate the general-purpose solvers.

In the literature, the technique of quadratic convex reformulation (QCR) has been widely applied to various mixed-integer quadratic programming problems, which can accelerate a general-purpose solver like Gurobi and CPLEX effectively. To our knowledge, the first convex reformulation for \eqref{BQP} was proposed by \cite{Hammer}, in which the objective function $\frac{1}{2}x^\top Q x+ c^\top x$ is replaced by $\frac{1}{2}x^\top Q x+ c^\top x-\lambda_{\min} \sum_{i=1}^n (x_i^2-x_i)$, where $\lambda_{\min}$ denotes the smallest eigenvalue of $Q$. Leveraging the semidefinite programming technique, Billionnet and Elloumi \cite{Billionnet2007} proposed a new quadratic convex reformulation for \eqref{BQP}, which yields a continuous relaxation bound that can be as tight as the conventional semidefinite relaxation bound of \eqref{BQP}. Billionnet et al. \cite{Billionnet2012,Billionnet2013} further proposed some enhanced quadratic convex reformulations, which introduce some additional variables in extended space. Applying the techniques proposed in \cite{Billionnet2012,Billionnet2013} to \eqref{BQP}, an enhanced quadratic convex reformulation can be derived, which yields  a continuous relaxation bound that is as tight as the SDP+RLT relaxation \cite{Anstreicher}  of \eqref{BQP}. Besides of \eqref{BQP}, the QCR technique has also been applied to various mixed integer optimization problems (for example, see \cite{Billionnet2016,Billionnet2017,Elloumi,Lu2025,Plateau,Wu,Zheng2020,Zheng2014}, we only name a few).

The main idea behind QCR is to derive a reformulation of the original problem that yields a high-quality convex continuous relaxation bound. Thus, the reformulation should be a mixed-integer convex quadratic optimization problem. Such a reformulation technique limits the flexibility of adding nonconvex constraints into the problem formulation. Actually, in modern general-purpose solvers such as Gurobi and SCIP, the continuous relaxation is not the unique choice for computing lower bounds. When solving a mixed-integer quadratic programming problem with nonconvex objective function and nonconvex quadratic constraints using these solvers, the lower bounds computed by the solvers are generally based on the McCormick relaxation \cite{McCormick}. Thus, in order to accelerate the solver, it is possible to derive nonconvex reformulations which provide tighter relaxation bounds than the continuous relaxation bounds of the convex reformulations, while these bounds can be computed as efficiently as the continuous relaxation bounds of the convex reformulations.



Leveraging the new advancements in modern solvers, the quadratic nonconvex reformulation (QNR) technique has been proposed in \cite{QNR}. The main idea behind the QNR technique is to reformulate a mixed-integer quadratic optimization problem to another one, which is not necessarily convex, for the purpose of providing a high quality McCormick relaxation bound. As an alternative reformulation technique, QNR provides more flexibility than QCR, as it can add nonconvex quadratic constraints into the problem formulation.

The QNR technique has been applied to some nonconvex quadratic programming problems with continuous variables \cite{QNR}. However, this technique has not been applied to a problem with integer variables. Based on the QNR technique, this paper proposes some new reformulations for \eqref{BQP}, and investigates their impacts to the McCormick relaxation bounds implemented in Gurobi. Furthermore, the QNR based reformulations are compared with the existing QCR based reformulations. Our results provide some new insights into how to design effective reformulations of mixed-integer quadratic programming problems for the purpose of accelerating a general-purpose solver.

This paper is arranged as follows. Section~\ref{sec2} provides a brief review of the quadratic convex reformulations of \eqref{BQP} that have been proposed in the literature. Section~\ref{sec3} introduces the concept of the quadratic nonconvex reformulation technique, proposes a general framework for tightening the McCormick relaxation bound by incorporating valid quadratic inequalities into the reformulation, and discusses the application of the triangle inequalities in the reformulation. Section~\ref{sec4} reports numerical results on comparing the proposed quadratic nonconvex reformulations with the existing quadratic convex reformulations.

The following notations are used in this paper: $\mathbb{R}^n$ denotes the set of $n$-dimensional real vectors, $\mathbb{R}^n_+$ denotes the set of nonnegative $n$-dimensional real vectors, $\mathbb{S}^n$ denotes the set of $n\times n$ real symmetric matrices, $\mathbb{X}^n$ denotes the set of $n\times n$ symmetric matrices with zero diagonal entries, and  $\mathbb{X}_+^n$ denotes the set of matrices in $\mathbb{X}^n$ with nonnegative entries.
For a vector $\lambda\in\mathbb{R}^n$, we use $\textrm{diag}(\lambda)$ to denote the $n\times n$ diagonal matrix with diagonal elements being equal to the entries of $\lambda$. For two matrices $A$, $B\in\mathbb{S}^n$, define $A\cdot B := \textrm{trace}(AB)$ as the matrix inner product, and $A\succeq B$ means that $A-B$ is positive semidefinite. $e\in\mathbb{R}^n$ denotes the vector with all entries equaling one.

\section{A review of quadratic convex reformulation}\label{sec2}

Before introducing the concept of quadratic nonconvex reformulation, we first review the quadratic convex reformulations of \eqref{BQP} that have been proposed in \cite{Billionnet2007,Billionnet2012,Billionnet2013} respectively. 

The reformulation proposed in \cite{Billionnet2007} is defined as follows:
\begin{equation}\label{QCR}\tag{QCR}
	\begin{aligned}
		\min ~&~ \frac{1}{2}x^\top Q x+ c^\top x+\sum_{i=1}^n\lambda_i (x_i^2-x_i) \\
		\text{s.t.} ~&~ x_i \in \{0,1\},~i=1,\ldots,n,
	\end{aligned}
\end{equation}
where $\lambda:=(\lambda_1,\ldots,\lambda_n)\in\mathbb{R}^n$ such that $Q+2\textrm{diag}(\lambda)\succeq 0$. Note that the objective function in \eqref{QCR} is convex, so that its continuous relaxation is convex and can be solved efficiently. The best parameter $\lambda^\ast$ for \eqref{QCR} is defined as the one that maximizes the continuous relaxation bound of the problem and is computed by solving the following semidefinite programming problem  \cite{Billionnet2007}:
\begin{equation}\label{SDPDual}
	\begin{aligned}
		\max ~&\tau\\
		\mbox{s.t.} ~&\begin{bmatrix}
			-2 \tau ~&~ (c -\lambda)^\top \\
			c -\lambda  ~&~ Q+2\textrm{diag}(\lambda)  \\
		\end{bmatrix}\succeq 0,\\
		~&\tau\in\mathbb{R},~ \lambda\in\mathbb{R}^n.
	\end{aligned}
\end{equation}
Note that the dual problem of \eqref{SDPDual} is given as follows:
\begin{equation}\label{SDP}\tag{SDP}
	\begin{aligned}
		\min ~&~ \frac{1}{2}Q\cdot X+c^\top x\\
		\mbox{s.t.} ~&X_{ii}=x_i,~i=1,\ldots,n\\
		&X\succeq xx^\top,
	\end{aligned}
\end{equation}
which is the conventional semidefinite relaxation of \eqref{BQP}. When using the optimal solution of \eqref{SDPDual} to construct \eqref{QCR}, the continuous relaxation bound of the reformulated problem is equal to the optimal value of \eqref{SDP}.

Another reformulation of \eqref{BQP} is proposed in \cite{Billionnet2012,Billionnet2013}, which introduces some additional variables in an extended space. Actually, the reformulations proposed in \cite{Billionnet2012,Billionnet2013} are derived for general mixed-integer quadratic optimization problems, which include \eqref{BQP} as a special case. Applying these reformulation techniques to \eqref{BQP}, the following reformulation can be derived:
\begin{equation}\label{QCRE}\tag{QCRE}
	\begin{aligned}
		\min ~&~ \frac{1}{2} x^\top (Q+2\textrm{diag}(\lambda)-2Z) x+c^\top x - \lambda^\top x + Z \cdot X \\
		\mbox{s.t.} ~~&(X_{ij},x_i,x_j)\in \mathcal{M}_{ij},~i,j=1,\ldots,n,~i\neq j,\\
		&X\in \mathbb{X}^n,~x\in\{0,1\}^n,
	\end{aligned}
\end{equation}
where $\lambda\in\mathbb{R}^n$ and $Z\in\mathbb{X}^n$ such that $Q+2\textrm{diag}(\lambda)-2Z\succeq 0$, and
\begin{equation}\label{eqn:McCormick}
	\begin{aligned}
		\mathcal{M}_{ij}:=\left\{ (X_{ij},x_i,x_j)\in\mathbb{R}^3 \left|\begin{array}{@{}llll}
			&X_{ij}\geq 0,X_{ij}-x_i-x_j+1\geq 0,\\
			&X_{ij}-x_i\leq 0,X_{ij}-x_j\leq 0,
		\end{array}\right.\right\}.
	\end{aligned}
\end{equation}
For any feasible solution $(x,X)$ of \eqref{QCRE}, it is easy to check that the equation $X_{ij}=x_i x_j$ $(\forall i\neq j)$ must hold under the constraints of the problem. 
Furthermore, under the condition that $Q+2\textrm{diag}(\lambda)-2Z\succeq 0$, the continuous relaxation of \eqref{QCRE} is convex. The parameters $\lambda^\ast$ and $Z^\ast$ that maximize the continuous relaxation bound of \eqref{QCRE} can be computed by solving the following problem \cite{Billionnet2012,Billionnet2013}:
\begin{equation}\label{SDPRLTDual}
	\begin{aligned}
		\max ~& \tau\\
		\mbox{s.t.} ~& \begin{bmatrix}
			-2 e^\top N e -2 \tau ~&~ (c -\lambda+2Ne -Re-Se)^\top \\
			c -\lambda+2Ne -Re-Se  ~&~ Q+2\textrm{diag}(\lambda)-2M-2N+2R+2S  \\
		\end{bmatrix}\succeq 0,\\
		~&\tau\in\mathbb{R},~ \lambda\in\mathbb{R}^n, ~ M,N,R,S\in\mathbb{X}^+_n,
	\end{aligned}
\end{equation}
and setting
\begin{equation}\label{eq_solution}
	\lambda^\ast=\lambda^\dag~\textrm{and}~Z^\ast= M^\dag +  N^\dag -  R^\dag -  S^\dag,
\end{equation}
where  $(\lambda^\dag,\tau^\dag,M^\dag,N^\dag,R^\dag,S^\dag)$ denotes the optimal solution of \eqref{SDPRLTDual}.
Note that the dual problem of \eqref{SDPRLTDual} is the well-known SDP+RLT relaxation \cite{Anstreicher}:
\begin{equation}\label{SDPRLT}\tag{SDP+RLT}
	\begin{aligned}
		\max ~&~ \frac{1}{2}Q\cdot X+c^\top x\\
		\mbox{s.t.} ~&X_{ii}=x_i,~i=1,\ldots,n\\
		&(X_{ij},x_i,x_j)\in \mathcal{M}_{ij},~i,j=1,\ldots,n,~i\neq j,\\
		&X\succeq xx^\top.
	\end{aligned}
\end{equation}
Thus, when using the best parameters $\lambda^\ast$ and $Z^\ast$ to derive \eqref{QCRE}, the continuous relaxation bound of \eqref{QCRE} is equal to the optimal value of \eqref{SDPRLT}.

As we can see, in both reformulations \eqref{QCR} and \eqref{QCRE}, the purpose is to derive some reformulations that yield high-quality convex continuous relaxation bounds, based on which the efficiency of a branch-and-bound algorithm implemented in a solver such as Gurobi can be improved. However, we remark that in modern solvers such as Gurobi, the continuous relaxation bound is not the only choice for a branch-and-bound algorithm. It is possible to derive nonconvex reformulations that provide computational efficient McCormick relaxation bounds tighter than the continuous relaxation bounds of \eqref{QCR} and \eqref{QCRE}, which may further speed up the solver.

\section{Quadratic nonconvex reformulations}\label{sec3}

When solving a mixed-integer quadratic programming problem with nonconvex quadratic constraints using a modern solver such as Gurobi, the lower bound method is based on the McCormick relaxation. We propose some new quadratic nonconvex reformulations of \eqref{BQP}, which provide high-quality McCormick relaxation bounds to the solver. Such new reformulations are more flexible on adding nonconvex quadratic constraints into the problem formulation, which can further enhance the lower bounds of the reformulations by exploiting valid quadratic inequalities. In this section, we first provide a brief introduction to the lower bound methods implemented in Gurobi. After introducing the lower bound methods, we propose several quadratic nonconvex reformulations for \eqref{BQP}.

\subsection{The lower bound methods implemented in Gurobi}\label{sec31}

In this subsection, we provide a brief introduction to the lower bound methods implemented in Gurobi.
Consider the following quadratically constrained quadratic programming problem with binary variables:
\begin{equation}\label{MIQP}\tag{MIQP}
	\begin{aligned}
		\min ~&~q_0(x) \\
		\mbox{s.t.} ~&~q_i(x)\leq b_i,~i=1,2,\ldots,m,\\
		~&~ x_i\in \{0,1\},~i=1,\ldots,n,
	\end{aligned}
\end{equation}
where $q_i(x):=\frac{1}{2}x^{T}Q_i x+c_i^{T}x$, $i=0,\ldots,n$.
When solving \eqref{MIQP} using Gurobi, the lower bound method is implemented as follows: First, the binary constraint $x_i\in \{0,1\}$ is relaxed to $0\leq x_i \leq 1$. Then, for each $i=0,\ldots,m$, the solver detects whether $q_i(x)$ is convex. If all these quadratic functions are convex, then the continuous relaxation of the problem is already convex, and the lower bound is computed by solving the continuous relaxation of the problem directly. Otherwise, for the cases that there exist some $i\in\{0,\ldots,m\}$ such that $q_i(x)$ is nonconvex, the nonconvex function is linearized to $\hat{q}_i(x,X):=\frac{1}{2}Q_i \cdot X+c_i^{T}x$, where $X:=xx^T$. Finally, the constraint $X:=xx^T$ is relaxed to $(X_{ij},x_i,x_j)\in\mathcal{M}_{ij}$ for each pair of $i\neq j$, together with the constraints $X_{ii}\geq x_i^2$ and $X_{ii} \leq x_i$ for  $i\in\{1,\ldots,n\}$. Since the inequalities defined in \eqref{eqn:McCormick} are just the McCormick inequalities derived in \cite{McCormick}, the relaxation  is called McCormick relaxation. Denote by $\mathcal{P}\subseteq\{1,\ldots,m\}$ the set of indices $i$ such that $Q_i\succeq 0$. The McCormick relaxation of \eqref{MIQP} is given as follows:

\begin{equation}\label{MCR}
	\begin{aligned}
		\min ~&~ \tilde{q}_0(x,X) \\
		\mbox{s.t.} ~&q_{i}(x)\leq b_i,~i\in\mathcal{P},\\
		&\hat{q}_{i}(x,X)\leq b_i,~i\in\{1,\ldots,m\}\setminus \mathcal{P},\\
		&x_i^2\leq X_{ii} \leq x_i,~i=1,\ldots,n,\\
		&(X_{ij},x_i,x_j)\in \mathcal{M}_{ij}, ~i\neq j, ~i,j=1,\ldots,n,\\
		&0\leq x_i\leq 1,~i=1,\ldots,n.
	\end{aligned}
\end{equation}
where $\tilde{q}_0(x,X)=q_0(x)$ if $Q_0\succeq 0$, and $\tilde{q}_0(x,X)=\hat{q}_0(x,X)$ otherwise. Note that the constraint $0\leq x_i\leq 1$ is redundant under $x_i^2\leq X_{ii} \leq x_i$ and can be dropped. The lower bound implemented in Gurobi is computed by solving \eqref{MCR}.

\begin{remark}
	The lower bound method introduced in this section is the default one in Gurobi. For more details on the lower bound methods implemented in Gurobi, we may refer to the document in the official website of Gurobi\footnote{{https://www.gurobi.com/events/non-convex-quadratic-optimization/}}.
\end{remark}

\begin{remark}\label{remark:gurobi}
	Two details in Gurobi should be explained. First, in the McCormick relaxation \eqref{MCR}, if all the nonzero coefficients for variables $X_{ij} ~(i\neq j)$ have the same sign, then the solver only introduce two inequalities rather than all the four inequalities to represent $(X_{ij},x_i,x_j)\in \mathcal{M}_{ij}$, since there must exist two redundant constraints that can be dropped, depends on the sign of these coefficients. Second, when \eqref{MIQP} contains equality constraint of the form $\frac{1}{2}x^{T}Q_i x+c_i^{T}x=b_i$, the constraint is also linearized to $\frac{1}{2}Q_i \cdot X+c_i^{T}x=b_i$.
\end{remark}

If \eqref{BQP} is reformulated as a mixed-integer quadratic programming problem with nonconvex quadratic constraints, then the McCormick relaxation is applied to handle the nonconvex constraints. The purpose of quadratic nonconvex reformulation is to derive nonconvex reformulations that have high quality McCormick relaxation bounds.

In the conventional quadratic convex reformulation technique, the reformulated problem must be a mixed-integer convex quadratic programming problem. This approach restricts the potential to develop more effective reformulations that yield tighter lower bounds. In contrast, quadratic nonconvex reformulation does not limit the reformulated problem to being a mixed-integer convex one. This technique thus offers greater flexibility in deriving a broader range of reformulations, enabling the provision of high-quality and computationally efficient lower bounds to the solver.


\subsection{A simple quadratic nonconvex reformulation for \eqref{BQP}}


Following the idea of quadratic nonconvex reformulation, we first derive a nonconvex reformulation as follows:
\begin{equation}\label{QNR}\tag{QNR}
	\begin{aligned}
		\min ~&~ \frac{1}{2}x^\top(Q+2\textrm{diag}(\lambda)-2Z) x+c^\top x -\lambda^\top x + w \\
		\textrm{s.t.}~&~ x_i \in \{0,1\},~i=1\ldots,n,\\
		~&~ w\geq x^\top Z x.
	\end{aligned}
\end{equation}
where $\lambda\in\mathbb{R}^n$ and $Z\in\mathbb{X}^n$, such that $Q+2\textrm{diag}(\lambda)-2Z\succeq 0$. Note that if $Z\neq 0$, then it is not positive semidefinite, since it has zero diagonal entries. Thus, based on the discussions in Section~\ref{sec31}, the constraint $w\geq x^\top Z x$ is relaxed to $w\geq Z\cdot X$, and the McCormick inequalities on $(X,x)$ are added.  When feeding problem formulation \eqref{QNR} into Gurobi, the lower bound for the root-node is computed by solving the following problem:
\begin{equation}\label{MCR-QNR}
	\begin{aligned}
		\min ~&~ \frac{1}{2}x^\top(Q+2\textrm{diag}(\lambda)-2Z) x+c^\top x -\lambda^\top x + w \\
		\textrm{s.t.}~&~ 0\leq x_i \leq 1,~i=1\ldots,n,\\
		&(X_{ij},x_i,x_j)\in \mathcal{M}_{ij}, ~i\neq j, ~i,j=1,\ldots,n,\\
		&~ w\geq  Z\cdot X.
	\end{aligned}
\end{equation}
We can check that \eqref{MCR-QNR} is equivalent to the continuous relaxation of \eqref{QCRE}. Thus, the best parameter $(\lambda^\ast,Z^\ast)$ of \eqref{QNR} can also be obtained by solving \eqref{SDPRLTDual} and computed by using \eqref{eq_solution}. When feeding \eqref{QNR} and \eqref{QCRE} into Gurobi, the two reformulations will have the same lower bound in their root-nodes if they are constructed using the same parameter $(\lambda^\ast,Z^\ast)$.

\begin{remark}\label{remark_qcre_qnr}
	We remark that although \eqref{QCRE} and \eqref{QNR} achieve the same lower bound in their root-nodes, the two reformulations are essentially different. The main difference is that some additional variables $X_{ij}$ for $i,j\in\{1,\ldots,n\}$ where $Z_{ij}\neq 0$ are explicitly introduced in the formulation of \eqref{QCRE}, while only one additional variable $w$ is introduced in \eqref{QNR}. When solving \eqref{QNR} using Gurobi, although the same set of additional variables $X_{ij}$ are potentially introduced for constructing the McCormick relaxation, these additional variables are dynamically maintained by the solver: As the branch-and-bound algorithm runs, some of the binary variables $x_i$ will be fixed to $0$ or $1$ and will be regarded as constants. Thus, if the value of one of the binary variable $x_i$ or $x_j$ is decided by the branching rule, the variable $X_{ij}$ is not introduced again. On the other hand, in \eqref{QCRE}, the set of additional variable are introduced statically, and are kept as the branch-and-bound algorithm runs. Thus, it is possible that \eqref{QNR} can be more effective than \eqref{QCRE}, since 
	the computational cost for handing a node in deep layer can be reduced by eliminating unnecessary variables.
\end{remark}

\subsection{Enhancing quadratic nonconvex reformulations using valid inequalities}

To derive a new reformulation that provides tighter lower bound than \eqref{SDPRLT}, we  exploit some valid quadratic inequalities to enhance the reformulation.

Assume that the following quadratic constraints are valid to the feasible domain of \eqref{BQP}:
\begin{equation}\label{gtx}
	g_t(x):=\frac{1}{2}x^\top A_t x+ c_t^\top x+b_t\leq 0,~t=1,\ldots,m.
\end{equation}
Our idea is to perturb the objective function using a new term $\sum_{t=1}^m \gamma_t (g_t(x)-v_t)$, where $\gamma_t\geq 0~(t=1,\ldots,m)$ are given parameters, and $v_t\in\mathbb{R}~(t=1,\ldots,m) $ are new variables that are constrained by $v_t= g_t(x)$. The perturbed objective function is given as
\begin{equation}
	\frac{1}{2}x^\top(Q+2\textrm{diag}(\lambda)-2Z)x+c^\top x-\lambda^\top x+w+\sum_{t=1}^m \gamma_t \left(g_t(x)-v_t\right),
\end{equation}
where $\lambda\in \mathbb{R}^n$, $Z\in\mathbb{X}^n$, and $\gamma\in\mathbb{R}^m_+$ are given parameters. We assume that $Q+2\textrm{diag}(\lambda)-2Z+\sum_{t=1}^m \gamma_t A_t\succeq 0$ holds. The reformulation is given as follows:
\begin{equation}\label{QNRE}\tag{QNRE}
	\begin{aligned}
		\min ~&\frac{1}{2}x^\top(Q+2\textrm{diag}(\lambda)-2Z)x+c^\top x-\lambda^\top x+w+\sum_{t=1}^m \gamma_t \left(g_t(x)-v_t\right)\\
		\textrm{s.t.}~&x_i \in \{0,1\},~i=1\ldots,n,\\
		&w\geq x^\top Z x,\\
		& v_t= g_t(x),~v_t\leq 0,~t=1,\ldots,m.
	\end{aligned}
\end{equation}
Note that since $g_t(x)\leq 0~(t=1,\ldots,m)$ holds for any $x\in \{0,1\}^n$, the constraint $v_t\leq 0$ is implied by $v_t= g_t(x)$ and is redundant. We emphasize that this redundant constraint should be kept in the reformulation, since it becomes effective on tightening the McCormick relaxation implemented in the solver.

Based on the discussions in Section~\ref{sec31}, when solving \eqref{QNRE} using Gurobi, the McCormick relaxation of the problem is given as follows:
\begin{equation}\label{QCRE-MCR}
	\begin{aligned}
		\min ~&\frac{1}{2}x^\top(Q+2\textrm{diag}(\lambda)-2Z)x+c^\top x-\lambda^\top x+w+\sum_{t=1}^m \gamma_t \left(g_t(x)-v_t\right)\\
		\textrm{s.t.}~ &w\geq Z\cdot X,\\
		&x_i^2\leq X_{ii} \leq x_i,~i=1,\ldots,n,\\
		& v_t= \frac{1}{2}A_t \cdot X+ c_t^\top x+b_t,~t=1,\ldots,m,\\
		&v_t\leq 0,~t=1,\ldots,m,\\
		&(X_{ij},x_i,x_j)\in \mathcal{M}_{ij}, ~i\neq j, ~i,j=1,\ldots,n.\\
	\end{aligned}
\end{equation}

The best parameter $(\lambda^\ast,Z^\ast,\gamma^\ast)$ of \eqref{QNRE} is defined as the one that maximizes the McCormick relaxation bound of the problem. We show that the best parameter can be computed by solving the following problem:
\begin{equation}\label{QNR-dual1}
	\begin{aligned}
		\max ~&~ \sigma \\
		\mbox{s.t.} ~&\begin{bmatrix}
			2\hat{b}-2\sigma &\hat{c}^T\\
			\hat{c} ~&\hat{Q} \\
		\end{bmatrix}\succeq 0,\\
		&\hat{Q}=Q+2\textrm{diag}(\lambda)+\sum_{t=1}^m\gamma_t A_t-2M-2N+2R+2S,\\
		&\hat{c}=c-\lambda+\sum_{t=1}^m \gamma_t c_t+2Ne-Re-Se,\\
		&\hat{b}=\sum_{t=1}^m \gamma_t b_t-\sum_{i=1}^n\sum_{j=1}^n N_{ij},\\
		&\sigma\in\mathbb{R},~\lambda\in\mathbb{R}^n,~\gamma\in\mathbb{R}^m_+,~M,N,R,S\in\mathbb{X}_+^n.
	\end{aligned}
\end{equation}
We provide the following theorem, whose proof is based on the idea of formulating the Lagrangian dual problem of \eqref{QCRE-MCR} as a semidefinite programming problem. We remark that similar proof techniques have been widely adopted in the literature of quadratic convex reformulations (for example, see Theorem 2 in \cite{Billionnet2012}, and Theorem 1 in \cite{Billionnet2013}, we only name a few).

\begin{theorem}\label{thm1}
	The parameter $(\lambda^\ast,Z^\ast,\gamma^\ast)$ that maximizes the McCormick relaxation bound of \eqref{QNRE} can be obtained by setting $\lambda^\ast=\lambda^\dag$, $\gamma^\ast=\gamma^\dag$, and
	\begin{equation}
		Z^\ast= M^\dag +  N^\dag -  R^\dag -  S^\dag,
	\end{equation}
	where $(\lambda^\dag,\gamma^\dag,\tau^\dag,M^\dag,N^\dag,R^\dag,S^\dag)$ denotes the optimal solution of \eqref{QNR-dual1}.
\end{theorem}
\begin{proof}
	First, we simplify \eqref{QCRE-MCR} as follows:
	\begin{equation}\label{QNRE-MCR}
		\begin{aligned}
			\min ~&~\frac{1}{2}x^\top(Q+2\textrm{diag}(\lambda)-2Z)x+c^\top x-\lambda^\top x+ Z\cdot X+\sum_{t=1}^m \gamma_t \left(g_t(x)-v_t\right)\\
			\textrm{s.t.}~ &x_i^2\leq X_{ii} \leq x_i,~i=1,\ldots,n,\\
			& v_t= \frac{1}{2}A_t \cdot X+ c_t^\top x+b_t,~t=1,\ldots,m\\
			&v_t\leq 0,~t=1,\ldots,m,\\
			&(X_{ij},x_i,x_j)\in \mathcal{M}_{ij}, ~i\neq j, ~i,j=1,\ldots,n.\\
		\end{aligned}
	\end{equation}	
	Define the Lagrangian function of \eqref{QNRE-MCR} as
	\begin{equation}
		\begin{aligned}
			&L(x,X,v,\alpha,\beta,\tau,\zeta,M,N,R,S)=\frac{1}{2}x^\top(Q+2\textrm{diag}(\lambda)-2Z)x+c^\top x-\lambda^\top x \\
			&\quad + Z\cdot X+\sum_{t=1}^m \gamma_t \left(\frac{1}{2} x^\top A_t x+c_t^\top x+b_t-v_t\right)+\sum_{i=1}^n\alpha_i(x_i^2-X_{ii})\\
			&\quad +\sum_{i=1}^n\beta_i(X_{ii}-x_i)+\sum_{t=1}^m \tau_t\left(  v_t-\frac{1}{2}A_t \cdot X- c_t^\top x-b_t \right)
			+\sum_{t=1}^m \zeta_t v_t\\
			&\quad +\sum_{i=1}^n\sum_{j=1}^n\left[ -M_{ij}X_{ij}-N_{ij}\left(X_{ij}-x_j-x_i+1\right)  \right]\\
			&\quad +\sum_{i=1}^n\sum_{j=1}^n\left[  R_{ij}(X_{ij}-x_i)+S_{ij}(X_{ij}-x_j) \right],
		\end{aligned}
	\end{equation}
	where $\alpha,\beta\in \mathbb{R}^n_+$, $\tau\in \mathbb{R}^m$ and $\zeta\in \mathbb{R}^m_+$ are the Lagrangian multipliers of the corresponding constraints in \eqref{QNRE-MCR}, and $M, N, R, S \in \mathbb{X}^n_+$, whose entries $M_{ij}, N_{ij}, R_{ij}, S_{ij}$ are the Lagrangian multipliers of the four McCormick inequalities appear in the constraint $\left(X_{ij},x_i,x_j\right)\in \mathcal{M}_{ij}$, respectively. Define the dual function of \eqref{QNRE-MCR} as
	\begin{equation}\label{dualfunction}
		p(\alpha,\beta,\tau,\zeta,M,N,R,S)=\min_{x\in\mathbb{R}^n,X\in\mathbb{S}^n,v\in\mathbb{R}^m}L(x,X,v,\alpha,\beta,\tau,\zeta,M,N,R,S),
	\end{equation}
	and define the associated dual problem of \eqref{QNRE-MCR} as
	\begin{equation}\label{QNR-dual}
		\begin{aligned}
			\max ~&p(\alpha,\beta,\tau,\zeta,M,N,R,S)\\
			\mbox{s.t.} ~~&\alpha,\beta\in \mathbb{R}^n_+,~\tau\in \mathbb{R}^m,~\zeta\in \mathbb{R}^m_+, ~M,N,R,S\in\mathbb{X}_+^n.
		\end{aligned}
	\end{equation}
	To avoid the unboundedness in minimizing the Lagrangian function in \eqref{dualfunction}, the coefficients for the matrix $X$ and vector $v$ in the Lagrangian function must be equal to zero. This leads to the following equations:
	\begin{equation}\label{eq_condition}
		\left\{
		\begin{aligned}
			&Z-\textrm{diag}(\alpha)+\textrm{diag}(\beta)-\sum_{t=1}^m\frac{1}{2} \tau_t A_t-M-N+R+S= 0,\\
			&-\gamma_t+\tau_t+\zeta_t=0,~t=1,\ldots,m.
		\end{aligned}\right.
	\end{equation}
	Using \eqref{eq_condition}, the Lagrangian function can be simplified to
	\begin{equation}
		L(x,X,v,\alpha,\beta,\tau,\zeta,M,N,R,S) = \frac{1}{2}x^T \hat{Q} x + \hat{c}^T x + \hat{b},
	\end{equation}
	where
	\begin{align}
		&\hat{Q}=Q+2\textrm{diag}(\lambda)-2Z+\sum_{t=1}^m \gamma_t A_t +2\textrm{diag}(\alpha),\\
		&\hat{c}=c-\lambda+\sum_{t=1}^m \gamma_t c_t-\beta- \sum_{t=1}^m \tau_t c_t +2Ne-Re-Se,\\
		&\hat{b}=\sum_{t=1}^m \gamma_t b_t-\sum_{t=1}^m \tau_t b_t-\sum_{i=1}^n\sum_{j=1}^n N_{ij}.
	\end{align}
	Based on the first equation in \eqref{eq_condition}, the expression for $\hat{Q}$ can equivalently be written as
	\begin{equation}
		\hat{Q}=Q+2\textrm{diag}(\lambda)+\sum_{t=1}^m\gamma_t A_t-\sum_{t=1}^m \tau_t A_t +2\textrm{diag}(\beta)-2M-2N+2R+2S.
	\end{equation}
	The dual problem can now be equivalently transformed to the following problem:
	\begin{equation}\label{QNR-dual2}
		\begin{aligned}
			\max ~&~ \sigma \\
			\mbox{s.t.} ~&~\frac{1}{2}x^T \hat{Q} x + \hat{c}^T x + \hat{b}-\sigma \geq 0,~\forall x\in\mathbb{R}^n,\\
			&\textrm{(16), (19--21)}\\
			&\sigma\in\mathbb{R},~\alpha\in \mathbb{R}^n_+,~\beta\in \mathbb{R}^n_+,~\tau\in \mathbb{R}^m,\zeta\in \mathbb{R}^m_+,~M,N,R,S\in\mathbb{X}_+^n.
		\end{aligned}
	\end{equation}
	Besides, the constraint $\frac{1}{2}x^T \hat{Q} x + \hat{c}^T x + \hat{b}-\sigma \geq 0~(\forall x\in\mathbb{R}^n)$ is equivalent to
	\begin{equation}\label{sdpconstraint}
		\begin{bmatrix}
			2\hat{b}-2\sigma &\hat{c}^T\\
			\hat{c} ~&\hat{Q} \\
		\end{bmatrix}\succeq 0
	\end{equation}
	Thus, the first constraint in \eqref{QNR-dual2} can be replaced by \eqref{sdpconstraint}.
	
	By strong duality, we have that the optimal value of \eqref{QNRE-MCR} is equal to the optimal value of \eqref{QNR-dual2}. Hence, the problem of maximizing the optimal value of \eqref{QNRE-MCR} is equivalent to the problem of maximizing the optimal value of \eqref{QNR-dual2}, which can be formulated as follows:
	\begin{equation}\label{QNR-dual3}
		\begin{aligned}
			\max ~&~ \sigma \\
			\mbox{s.t.} ~&\begin{bmatrix}
				2\hat{b}-2\sigma &\hat{c}^T\\
				\hat{c} ~&\hat{Q} \\
			\end{bmatrix}\succeq 0,\\
			&\hat{Q}=Q+2\textrm{diag}(\lambda)+\sum_{t=1}^m\gamma_t A_t-\sum_{t=1}^m \tau_t A_t +2\textrm{diag}(\beta)-2M-2N+2R+2S,\\
			&\hat{c}=c-\lambda+\sum_{t=1}^m \gamma_t c_t-\beta- \sum_{t=1}^m \tau_t c_t +2Ne-Re-Se,\\
			&\hat{b}=\sum_{t=1}^m \gamma_t b_t-\sum_{t=1}^m \tau_t b_t-\sum_{i=1}^n\sum_{j=1}^n N_{ij},\\
			&Z-\textrm{diag}(\alpha)+\textrm{diag}(\beta)-\sum_{t=1}^m\frac{1}{2} \tau_t A_t-M-N+R+S= 0,\\
			&-\gamma_t+\tau_t+\zeta_t=0,~t=1,\ldots,m,\\
			&Q+2\textrm{diag}(\lambda)-2Z+\sum_{t=1}^m \gamma_t A_t\succeq 0,\\
			&\sigma\in\mathbb{R},~\alpha\in \mathbb{R}^n_+,~\beta\in \mathbb{R}^n_+,~\tau\in \mathbb{R}^m,\zeta\in \mathbb{R}^m_+,~M,N,R,S\in\mathbb{X}_+^n,\\
			&\lambda\in\mathbb{R}^n,~\gamma\in\mathbb{R}^m_+,~Z\in\mathbb{X}^n.
		\end{aligned}
	\end{equation}
	
	From any feasible solution $(\sigma,\alpha,\beta,\tau,\zeta,M,N,R,S,\lambda,\gamma,Z)$ of \eqref{QNR-dual3}, we can construct another solution $(\sigma,\alpha^\prime,\beta^\prime,\tau^\prime,\zeta,M,N,R,S,\lambda^\prime,\gamma^\prime,Z^\prime)$ by setting $\alpha^\prime=0$, $\beta^\prime=0$, $\tau^\prime=0$, $\gamma^\prime=\gamma-\tau$, $\lambda^\prime=\lambda+\beta$, and $Z^\prime=M+N-R-S$. We can check that the constructed solution is also feasible to \eqref{QNR-dual3}, and has the same objective value as the original solution. Thus, we may fix the vectors $\alpha$, $\beta$, and $\tau$ to zero to show that \eqref{QNR-dual3} is equivalent to \eqref{QNR-dual1}, and the first equation in  \eqref{eq_condition} is reduced to $Z=M+N-R-S$.
\end{proof}

We further remark that the dual problem of \eqref{QNR-dual1} is
\begin{equation}\label{SDPRLTVI}\tag{SDP+RLT+VI}
	\begin{aligned}
		\max ~&~ \frac{1}{2}Q\cdot X+c^\top x\\
		\mbox{s.t.} ~&X_{ii}=x_i,~i=1,\ldots,n\\
		&(X_{ij},x_i,x_j)\in \mathcal{M}_{ij},~i,j=1,\ldots,n,~i\neq j,\\
		&\frac{1}{2}A_t\cdot X+c_t^\top x+b_t\leq 0,~t=1,\ldots,m,\\
		&X\succeq xx^\top.
	\end{aligned}
\end{equation}

Although \eqref{QNRE} provides a high quality McCormick relaxation bound, its main weakness is that the number of extended variables $\{v_t\}_{t=1}^m$ is generally very large, which significantly increases the computational cost of solving the McCormick relaxation. In order to reduce the number of extended variables in the McCormick relaxation, we simplify the formulation of \eqref{QNRE} as follow: First, we replace the constraint $v_t\leq 0$ by $g_t(x)\leq 0$. Second, we introduce a single variable $v\in\mathbb{R}$ to represent the aggregated value $\sum_{t=1}^n \gamma_t v_t$, and introduce the constraint $v=\sum_{t=1}^n \gamma_t g_t(x)$. Finally, the extended variables $\{v_t\}_{t=1}^m$ are dropped. The simplified formulation is given as follows:
\begin{equation}\label{QNREAGG}\tag{QNRE-AGG}
	\begin{aligned}
		\min ~&\frac{1}{2}x^\top(Q+2\textrm{diag}(\lambda)-2Z)x+c^\top x-\lambda^\top x+w+\sum_{t=1}^m \gamma_t g_t(x)-v\\
		\textrm{s.t.}~&x_i \in \{0,1\},~i=1\ldots,n,\\
		&w\geq x^\top Z x,\\
		&v\leq \sum_{t=1}^n \gamma_t g_t(x)\\
		&g_t(x)\leq 0,~t=1,\ldots,m.
	\end{aligned}
\end{equation}
Such reformulation aggregates the set of variables $\{v_t\}_{t=1}^m$ into just one variable, which can reduce the number of variables significantly when $m$ is huge. Furthermore, in \eqref{QNREAGG}, the constraint $v=\sum_{t=1}^n \gamma_t g_t(x)$ is relaxed to $v\leq \sum_{t=1}^n \gamma_t g_t(x)$. Such relaxation does not change the optimal value of the problem, but may further reduce the number of inequalities introduced in the McCormick relaxation of the problem, as discussed in Remark~\ref{remark:gurobi}.

Finally, we show that the McCormick relaxations of \eqref{QNREAGG} and \eqref{QNRE} are equivalent. We need the following assumption:

\begin{assumption}
Assume that $A_t\in\mathbb{X}^n$ holds for all $t=1,\ldots,m$, where $A_t$ denotes the matrix in the quadratic term of the function $g_t(x)$ defined in \eqref{gtx}.
\end{assumption}

Assumption 1 can be satisfied easily for a 0-1 integer programming problem, since all the quadratic terms $x_i^2$ in $g_t(x)$ can be replaced by the linear term $x_i$.

Under Assumption 1, we can show that $g_t(x)~(t=1,\ldots,m)$ is nonconvex if $A_t\neq 0$, and $\sum_{t=1}^m \gamma_t g_t(x)$ is nonconcave if $\sum_{t=1}^m \gamma_t A_t\neq 0$. Thus, all the quadratic constraints in \eqref{QNREAGG} are either nonconvex constraints or reduced to linear constraints. Based on the discussions in Section~\ref{sec31}, when feeding \eqref{QNREAGG} into Gurobi, all these quadratic constraints will be linearized (or are already linear ones) in the McCormick relaxation. It is easy to check that the McCormick relaxation of \eqref{QNREAGG} is as tight as the one of \eqref{QNRE}. Thus, \eqref{QNREAGG} provides the same lower bound as \eqref{QNRE}, but contains much fewer variables than \eqref{QNRE}.

It is interesting to see that the valid constraints $g_t(x)\leq 0 ~(t=1,\ldots,m)$ are presented in \eqref{QNREAGG} directly. Such constraints will be linearized by the solver, which can be effective for tightening the McCormick relaxation. This idea provides us a new insight into designing effective reformulations: Adding some redundant nonconvex quadratic constraints into the problem formulation directly can be effective on accelerating a solver. Such insight is counter-intuitive: In previous researches, for a mixed-integer programming problem, the valid inequalities added into a formulation are generally assumed to be convex. On the other hand, the nonconvex valid inequalities have only been applied previously for tightening the convex relaxations (such as semidefinite relaxations) of \eqref{BQP}, but have not been added into the problem formulation directly for the purpose of accelerating a solver.



Furthermore, for the cases where $m$ is huge, we only add the constraints $g_t(x)\leq 0$ where $\gamma^\ast_t>0$ into \eqref{QNREAGG} rather than all the valid constraints. Note that $\gamma_t^\ast$ denotes the parameter constructed in Theorem~\ref{thm1}. Thus, $\gamma^\ast_t>0$ means that $\frac{1}{2}A_t\cdot X+c_t^\top x+b_t\leq 0$ is active at the optimal solution of \eqref{SDPRLTVI}. Dropping inactive constraints in \eqref{SDPRLTVI} does not affect its optimal value.

\subsection{A specific reformulation by incorporating triangle inequalities}

In the literature, various valid inequalities have been proposed to tighten the semidefinite relaxation of a 0-1 integer programming problem. To illustrate the effectiveness of the proposed reformulation framework, we derive a specific quadratic nonconvex reformulation using the following inequalities named as (modified) triangle inequalities:
\begin{equation}\label{triangle}
	\left\{
	\begin{aligned}
		&x_ix_j + x_i x_k-x_j x_k-x_i\leq 0\\
		&x_ix_j + x_j x_k-x_i x_k-x_j\leq 0\\
		&x_ix_k + x_j x_k-x_i x_j-x_k\leq 0\\
		&x_i+x_j+x_k-x_i x_j -x_i x_k-x_jx_k-1\leq 0
	\end{aligned}\right.~\forall ~1\leq i<j<k\leq n.
\end{equation}
We remark that the original version of triangle inequalities, which are valid for $x\in\{-1,1\}^n$, are derived for the max-cut problem \cite{Helmberg,Rendl}. The inequalities presented in \eqref{triangle} are modified from the original triangle inequalities by variable substitutions and are valid for $x\in\{0,1\}^n$.


Based on Theorem~\ref{thm1}, using the set of triangle inequalities, we can derive a quadratic nonconvex reformulation in the form of \eqref{QNREAGG} whose McCormick relaxation bound can be as tight as the following problem:
\begin{equation}\label{SDP+RLT+TRI}\tag{SDP+RLT+TRI}
	\begin{aligned}
		\min ~&~ \frac{1}{2} Q\cdot X+c^\top x  \\
		\mbox{s.t.:} ~~&X_{ii}=x_{i},~i=1,\ldots,n,\\
		&(X_{ij},x_i,x_j)\in \mathcal{M}_{ij},~i,j=1,\ldots,n,~i\neq j,\\
		&(x,X)\in\mathcal{T},\\
		&X-xx^\top\succeq 0.
	\end{aligned}
\end{equation}
where $\mathcal{T}$ denotes the set of points $(x,X)$ satisfying the following constraints:
\begin{equation}\label{triangle2}
	\left\{
	\begin{aligned}
		&X_{ij} + X_{ik}-X_{jk}-x_i\leq 0\\
		&X_{ij} + X_{jk}-X_{ik}-x_j\leq 0\\
		&X_{ik} + X_{jk}-X_{ij}-x_k\leq 0\\
		&x_i+x_j+x_k-X_{ij} -X_{ik}-X_{jk}-1\leq 0
	\end{aligned}\right.~\forall ~1\leq i<j<k\leq n.
\end{equation}
In practical computations, \eqref{SDP+RLT+TRI} could be much tighter than both \eqref{SDPRLT} and \eqref{SDP}. 


\section{Numerical experiments}\label{sec4}

In this section, we evaluate the effectiveness of the proposed reformulations. We compare the proposed QNR based reformulations \eqref{QNR} and \eqref{QNREAGG} with the two QCR based reformulations \eqref{QCR} and \eqref{QCRE}.

Our numerical experiments are conducted on some public test instances downloaded from Biq Mac Library\footnote{https://biqmac.aau.at/biqmaclib.html}. We select 50 test instances with 100-150 variables and with density in the range of $30\%-100\%$. The numerical results based on these public test instances will be reported in Section~\ref{sec41}. In addition to the experiments on public test instances, in order to evaluate the performance of various reformulation methods on instances with different size and different densities, we further generate 60 test instances. The results on these randomly generated test instances will be reported in Section~\ref{sec42}.

To evaluate the effects of different reformulation methods, the following reformulations are compared:
\begin{enumerate}
	\item[$\bullet$] \textbf{QCR}: The quadratic convex reformulation \eqref{QCR}, whose parameters are determined by solving the dual problem of \eqref{SDP}.
	\item[$\bullet$] \textbf{QCRE}: The quadratic convex reformulation \eqref{QCRE}, whose parameters are determined by solving the dual problem of \eqref{SDPRLT}.
	\item[$\bullet$] \textbf{QNR}: The quadratic nonconvex reformulation \eqref{QNR}, whose parameters are determined by solving the dual problem of \eqref{SDPRLT}.
	\item[$\bullet$] \textbf{QNR+Tri}: The quadratic nonconvex reformulation \eqref{QNREAGG}, in which the set of valid inequalities are composed by all the triangle inequalities, and the parameters are determined by solving the dual problem of \eqref{SDP+RLT+TRI}.
\end{enumerate}
We solve all the reformulations using Gurobi, and investigate which reformulation is more effective in accelerating Gurobi by comparing the computational time.

All our experiments were conducted on a personal computer, with Intel i7-9700 CPU (3.00 GHz) and 16 GB RAM. Our codes were implemented in Matlab. All the semidefinite programming problems were solved using Mosek (ver 10.2.2), and all the reformulated problems were solved using Gurobi (ver 12.01). For fair comparisons, all tests used a single thread computation. The relative optimality tolerance of Gurobi was set to $10^{-4}$, the same as in \cite{Locatelli2024}. All other parameters in the solver were set to their default value.

When constructing a reformulation, a corresponding semidefinite programming problem is solved. Note that the two semidefinite programming problems \eqref{SDPRLT} and  \eqref{SDP+RLT+TRI} contain a large number of constraints. Similar to the method in \cite{QNR}, we solve the two problems using an iterative method described as follows:
\begin{itemize}
	\item[1.] Select \eqref{SDP} as the initial semidefinite relaxation problem and solve the initial problem for its optimal solution $(x^\ast,X^\ast)$. Set Iter$=0$.
	\item[2.] For the solution $(x^\ast,X^\ast)$, detect whether 
	the McCormick inequalities (and the triangle inequalities, together) are violated, and add these violated constraints into the semidefinite relaxation.
	\item[3.] Solve the updated problem again for its optimal solution. Update Iter$=$Iter$+1$.
	\item[4.] If the number of iterations does not reach the preset value, go to Step 2 to run the next iteration. Otherwise, we extract the dual solution of the current updated problem and construct the reformulation using the extracted dual solution.
\end{itemize}
This adaptive method saves a notable computational time while providing a high-quality near optimal solution of the semidefinite programming problems. The number of iterations is tuned to balance the trade-off between the quality of the bound and the computational time for solving the problem, which is set to 2 for \eqref{SDPRLT}, and is set to 9 for \eqref{SDP+RLT+TRI} in our experiments.

\subsection{Numerical results on public test instances from Biq Mac Library}\label{sec41}

We first study the performance of different reformulations on the test instances downloaded from Biq Mac Library. We select 50 instances. The first 10 instances, named as be100.$k$, where $k=1,\ldots,10$, are 100-dimensional test instances with full density. The other 40 instances are named as be$n$.$d$.$k$, where $n\in\{120,150\}$ represents the number of binary variables, $d\in\{3,8\}$ represents the density ($30\%$ and $80\%$, respectively), and $k\in\{1\ldots,10\}$ denotes the index of the instance. We remark that the 40 instances of dimensional 120-150 are also used in \cite{Locatelli2024} in their numerical experiments.

Then, we solve all these test instances using Gurobi, with four different reformulations. The time limit of Gurobi is set to 7200 seconds (which is also the setting in \cite{Locatelli2024}). For each instance, we record the time for preprocessing (i.e., the time of solving the semidefinite programming problem using Mosek) and the time for running branch-and-bound algorithm (i.e., the computational time recorded by Gurobi) separately. The number of nodes enumerated by the solver and the lower bound in the root-node are all recorded. Furthermore, if the instance is not solved within 7200 seconds, then we also record the final gap output by Gurobi.

The numerical results are presented in Table~\ref{tab1}. The column ``Node'' indicates the number of nodes, and the column ``Time'' indicates the computational time, which includes the total time of both preprocessing and running branch-and-bound algorithm. If an instance is not solved in the time limit, then we report the final relative gap in the column ``Node'', and mark ``7200'' in the column ``Time''.

In addition to Table~\ref{tab1}, more results are reported in Table~\ref{tab1b} to demonstrate the tightness of different reformulations and the computational cost for preprocessing. Actually, for each instance, we have checked that when using a specific reformulation method, the lower bound in the root-node (returned in the log file of Gurobi) is always equal to the lower bound of the corresponding semidefinite programming relaxation problem solved in the preprocessing step. Thus, we only list the relative gap (RG) of the semidefinite relaxation bounds, and the computational time for solving different semidefinite relaxations, including \eqref{SDP}, \eqref{SDPRLT}, and \eqref{SDP+RLT+TRI}. Note that the results of \eqref{SDPRLT} are shared by  both QCRE and QNR.

\begin{remark}
	The relative gap in Table~\ref{tab1b} is defined as $\textrm{RG}=(v^\ast-\ell)/\vert v^\ast \vert$, where $v^\ast$ denotes the optimal value of the problem, and $\ell$ denotes the lower bound returned by the semidefinite relaxation. This is different from the final gap returned by the solver, in which $v^\ast$ is defined as the upper bound (the objective value of the best-known solution, determined by the branch-and-bound algorithm) rather than the optimal value.
\end{remark}

\begin{table}[h!]
	\caption{Numerical results on Biq Mac test instances.}\label{tab1}
	\begin{tabular}{rrrrrrrrr}	
		\toprule
		\multirow{2}*{Instance} & \multicolumn{2}{c}{QCR} & \multicolumn{2}{c}{QCRE}&\multicolumn{2}{c}{QNR}&\multicolumn{2}{c}{QNR+TRI}\\
		\cmidrule{2-9}
		~ &Node & Time &Node & Time &Node & Time &Node & Time \\
		\midrule
		be100.1 &(4.09\%) &7200 &1066 &60 &1526 &52 &1 &40  \\
		be100.2 &(5.52\%) &7200 &2400 &102 &4421 &101 &1 &30  \\
		be100.3 &(5.40\%) &7200 &1104 &59 &1647 &45 &1 &36  \\
		be100.4 &(4.47\%) &7200 &3294 &157 &5258 &104 &63 &54  \\
		be100.5 &(8.02\%) &7200 &6390 &269 &10904 &132 &59 &58  \\
		be100.6 &(5.49\%) &7200 &2044 &100 &2133 &64 &1 &28  \\
		be100.7 &(6.69\%) &7200 &4519 &296 &6393 &440 &91 &69  \\
		be100.8 &(7.92\%) &7200 &28181 &1149 &36502 &347 &1 &21  \\
		be100.9 &(9.57\%) &7200 &20221 &562 &36017 &267 &291 &103  \\
		be100.10 &(8.31\%) &7200 &8511 &326 &10852 &111 &45 &87  \\
		be120.3.1 &1833 &197 &30361 &1635 &43998 &1427 &124 &75  \\
		be120.3.2 &973 &66 &2970 &226 &4751 &127 &1 &55  \\
		be120.3.3 &2183 &197 &11412 &437 &16373 &331 &1 &47  \\
		be120.3.4 &1409 &101 &8939 &535 &13379 &502 &1 &53  \\
		be120.3.5 &1400 &120 &5383 &322 &10057 &141 &1 &49  \\
		be120.3.6 &1218 &89 &2326 &158 &2435 &414 &1 &84  \\
		be120.3.7 &689 &16 &1253 &86 &1789 &70 &1 &58  \\
		be120.3.8 &1027 &58 &3163 &217 &5775 &315 &1 &58  \\
		be120.3.9 &17431 &2314 &52931 &1792 &91230 &876 &102 &85  \\
		be120.3.10 &2596 &264 &10703 &697 &16371 &264 &1 &48  \\
		be120.8.1 &(9.82\%) &7200 &(1.60\%) &7200 &537912 &6002 &6538 &1005  \\
		be120.8.2 &(8.48\%) &7200 &66995 &3838 &89197 &1810 &772 &280  \\
		be120.8.3 &(7.78\%) &7200 &35528 &2744 &49419 &672 &390 &141  \\
		be120.8.4 &(6.16\%) &7200 &6367 &538 &9845 &265 &30 &85  \\
		be120.8.5 &(5.85\%) &7200 &5337 &443 &10368 &243 &24 &94  \\
		be120.8.6 &(8.18\%) &7200 &45310 &3232 &53647 &767 &404 &220  \\
		be120.8.7 &(8.25\%) &7200 &(1.09\%) &7200 &272801 &2597 &2444 &723  \\
		be120.8.8 &(10.05\%) &7200 &(1.98\%) &7200 &(0.85\%) &7200 &16573 &1371  \\
		be120.8.9 &(9.23\%) &7200 &62874 &3308 &90931 &544 &738 &336  \\
		be120.8.10 &(6.69\%) &7200 &13592 &1055 &22802 &521 &229 &123  \\
		be150.3.1 &(6.76\%) &7200 &(0.65\%) &7200 &175736 &2233 &168 &163  \\
		be150.3.2 &(7.30\%) &7200 &(1.09\%) &7200 &301492 &4321 &366 &226  \\
		be150.3.3 &(5.98\%) &7200 &9912 &1045 &18633 &871 &1 &107  \\
		be150.3.4 &(5.55\%) &7200 &12561 &1197 &31397 &966 &1 &130  \\
		be150.3.5 &(7.96\%) &7200 &(1.05\%) &7200 &313252 &5126 &830 &304  \\
		be150.3.6 &(8.90\%) &7200 &(2.17\%) &7200 &(1.56\%) &7200 &15590 &1274  \\
		be150.3.7 &(7.92\%) &7200 &(0.88\%) &7200 &266032 &6842 &388 &219  \\
		be150.3.8 &(8.38\%) &7200 &(2.22\%) &7200 &(1.63\%) &7200 &30836 &3236  \\
		be150.3.9 &(13.01\%) &7200 &(4.32\%) &7200 &(3.72\%) &7200 &(0.75\%) &7200  \\
		be150.3.10 &(8.66\%) &7200 &(2.23\%) &7200 &(1.54\%) &7200 &26862 &2366  \\
		be150.8.1 &(9.13\%) &7200 &(1.91\%) &7200 &(0.88\%) &7200 &15377 &1577  \\
		be150.8.2 &(9.74\%) &7200 &(2.51\%) &7200 &(1.64\%) &7200 &(0.77\%) &7200  \\
		be150.8.3 &(7.69\%) &7200 &(1.45\%) &7200 &(1.02\%) &7200 &18023 &2162  \\
		be150.8.4 &(8.56\%) &7200 &(1.40\%) &7200 &(0.97\%) &7200 &5111 &983  \\
		be150.8.5 &(7.22\%) &7200 &(1.45\%) &7200 &416118 &5109 &5824 &1328  \\
		be150.8.6 &(9.75\%) &7200 &(1.90\%) &7200 &(1.25\%) &7200 &16801 &2464  \\
		be150.8.7 &(7.72\%) &7200 &(1.95\%) &7200 &(1.20\%) &7200 &45868 &5003  \\
		be150.8.8 &(7.75\%) &7200 &(1.55\%) &7200 &(1.06\%) &7200 &15050 &1770  \\
		be150.8.9 &(9.06\%) &7200 &(2.36\%) &7200 &(1.76\%) &7200 &(0.45\%) &7200  \\
		be150.8.10 &(8.10\%) &7200 &(1.56\%) &7200 &(1.02\%) &7200 &14709 &1381  \\
		\bottomrule
	\end{tabular}
\end{table}

\begin{table}[h!]
	\centering
	\caption{Root-node relative gaps on Biq Mac test instances.}\label{tab1b}
	\begin{tabular}{rrrrrrr}	
		\toprule
		\multirow{2}*{Instance} & \multicolumn{2}{c}{SDP} & \multicolumn{2}{c}{SDP+RLT} &\multicolumn{2}{c}{SDP+RLT+TRI}\\
		\cmidrule{2-3} \cmidrule{4-5} \cmidrule{6-7}
		~ &RG & Time &RG & Time &RG & Time \\
		\midrule
		be100.1 &5.31\% &0.1 &1.09\% &4.0 &0.00\% &33.1  \\
		be100.2 &6.24\% &0.0 &1.58\% &3.4 &0.00\% &22.4  \\
		be100.3 &6.62\% &0.1 &1.23\% &4.4 &0.00\% &30.2  \\
		be100.4 &5.24\% &0.1 &1.51\% &4.0 &0.08\% &26.3  \\
		be100.5 &9.00\% &0.1 &2.71\% &4.3 &0.16\% &18.4  \\
		be100.6 &6.72\% &0.1 &1.70\% &4.6 &0.00\% &23.0  \\
		be100.7 &7.91\% &0.1 &1.96\% &5.4 &0.08\% &24.1  \\
		be100.8 &8.95\% &0.1 &3.11\% &3.4 &0.63\% &20.1  \\
		be100.9 &10.77\% &0.1 &4.35\% &3.4 &0.65\% &17.8  \\
		be100.10 &9.49\% &0.0 &2.74\% &3.9 &0.10\% &17.7  \\
		be120.3.1 &8.25\% &0.1 &2.51\% &7.2 &0.17\% &43.5  \\
		be120.3.2 &6.58\% &0.1 &1.45\% &9.6 &0.00\% &50.6  \\
		be120.3.3 &6.61\% &0.1 &1.97\% &7.8 &0.00\% &41.1  \\
		be120.3.4 &6.43\% &0.1 &1.74\% &8.6 &0.00\% &44.0  \\
		be120.3.5 &7.58\% &0.1 &2.01\% &9.7 &0.00\% &45.1  \\
		be120.3.6 &6.38\% &0.1 &1.35\% &9.0 &0.00\% &77.5  \\
		be120.3.7 &5.14\% &0.1 &0.94\% &7.7 &0.00\% &54.8  \\
		be120.3.8 &5.13\% &0.1 &1.30\% &6.9 &0.00\% &50.8  \\
		be120.3.9 &10.32\% &0.1 &3.91\% &7.6 &0.22\% &30.7  \\
		be120.3.10 &8.41\% &0.1 &2.21\% &10.6 &0.00\% &38.8  \\
		be120.8.1 &10.51\% &0.1 &4.08\% &7.9 &1.24\% &33.8  \\
		be120.8.2 &9.07\% &0.1 &3.16\% &9.5 &0.79\% &34.1  \\
		be120.8.3 &8.46\% &0.1 &2.91\% &8.9 &0.54\% &36.6  \\
		be120.8.4 &6.95\% &0.1 &1.91\% &9.0 &0.02\% &42.1  \\
		be120.8.5 &6.47\% &0.1 &1.77\% &8.3 &0.02\% &45.0  \\
		be120.8.6 &8.97\% &0.1 &2.92\% &9.4 &0.50\% &34.5  \\
		be120.8.7 &8.99\% &0.1 &3.00\% &11.2 &0.91\% &45.3  \\
		be120.8.8 &10.81\% &0.1 &4.56\% &9.3 &1.57\% &33.6  \\
		be120.8.9 &9.65\% &0.1 &3.16\% &7.6 &0.59\% &35.2  \\
		be120.8.10 &7.27\% &0.1 &2.17\% &9.2 &0.32\% &43.3  \\
		be150.3.1 &6.81\% &0.1 &2.21\% &12.3 &0.15\% &98.3  \\
		be150.3.2 &7.79\% &0.1 &2.66\% &16.7 &0.28\% &102.1  \\
		be150.3.3 &6.37\% &0.1 &1.57\% &14.4 &0.00\% &99.5  \\
		be150.3.4 &5.87\% &0.1 &1.42\% &11.6 &0.01\% &101.4  \\
		be150.3.5 &8.29\% &0.1 &2.84\% &14.2 &0.47\% &92.6  \\
		be150.3.6 &9.36\% &0.1 &4.00\% &11.3 &1.20\% &68.5  \\
		be150.3.7 &8.34\% &0.1 &2.50\% &12.9 &0.34\% &84.4  \\
		be150.3.8 &9.03\% &0.1 &3.72\% &12.1 &1.32\% &83.6  \\
		be150.3.9 &13.41\% &0.1 &6.46\% &12.5 &2.34\% &64.6  \\
		be150.3.10 &8.72\% &0.1 &3.61\% &11.8 &1.25\% &89.6  \\
		be150.8.1 &9.53\% &0.2 &3.86\% &11.2 &1.26\% &76.2  \\
		be150.8.2 &9.99\% &0.2 &4.10\% &10.9 &1.67\% &63.2  \\
		be150.8.3 &7.72\% &0.1 &2.93\% &10.9 &1.04\% &83.3  \\
		be150.8.4 &9.31\% &0.1 &3.23\% &12.4 &0.87\% &76.8  \\
		be150.8.5 &7.64\% &0.1 &2.64\% &12.0 &0.78\% &92.0  \\
		be150.8.6 &9.92\% &0.1 &3.55\% &12.9 &1.16\% &76.6  \\
		be150.8.7 &7.81\% &0.1 &3.22\% &14.3 &1.33\% &90.1  \\
		be150.8.8 &8.16\% &0.1 &2.94\% &14.4 &0.99\% &89.8  \\
		be150.8.9 &9.64\% &0.1 &4.07\% &13.5 &1.50\% &79.9  \\
		be150.8.10 &8.15\% &0.1 &3.08\% &13.4 &0.98\% &93.2  \\
		\bottomrule
	\end{tabular}
\end{table}		

From the results in Table~\ref{tab1}, we can see that QNR+TRI performs better than the other three reformulations on almost all the instances (in terms of computational time, or the final gap if the solver fails to solve the instance). The main reason for the good performance of QNR+TRI is revealed in Table~\ref{tab1b}, which shows that the gap yielded by \eqref{SDP+RLT+TRI} is significantly smaller than the gaps of \eqref{SDP} and \eqref{SDPRLT}. Thus, using QNR+TRI, the number of nodes can be much fewer than that of the other three reformulations, due to the high quality lower bound. Although the preprocessing time of QNR+TRI is much longer than those of the other reformulations, the time saved by the branch-and-bound algorithm can be quite significant. Thus, the total time of QNR+TRI can be much shorter than those of the other three reformulations.

It is interesting to compare the performance of QNR and QCRE. Actually, the two reformulations share the same preprocessing step, and use the same parameters for constructing reformulations. However, their computational performances are quite different. From the results in Table~\ref{tab1}, we can see that QNR performs better than QCRE on most test instances. These results support our analysis in Remark~\ref{remark_qcre_qnr}: QCRE keeps all the extended variables in $X$ statically, but QNR keeps these extended variables dynamically. As the branch-and-bound algorithm runs, some variables in $x$ are fixed as constants. Then, QNR can dynamically eliminate some extended variables, so that the lower bound can be computed more efficiently than QCRE.

\begin{remark}\label{remark_qnr_qcre_root}
	We further checked the log file output by Gurobi to observe the performance of QCRE and QNR, and discover that QNR needs much longer time than QCRE for preprocessing the root-node. A possible reason is that Gurobi may need additional preprocessing time for generating the McCormick inequalities for the root-node. In comparison, since the continuous relaxation of QCRE is already convex, such preprocessing time is not needed. Although QNR needs longer time than QCRE for processing the root-node, it can be more efficient on handling nodes in deep layers.
\end{remark}

Furthermore, we can see that QCR fails to solve all the instances with high density (larger than or equal to $80\%$) or with more than 150 variables. However, it is very effective for the cases where $n=120$ and $d=3$. In these cases, QCR performs even better than both QNR and QCRE, although its lower bound in the root-node is not as tight as the bounds of the other two reformulations. Note that QCR only perturbs the diagonal entries of $Q$, which preserves the sparsity of the problem. Such sparsity may bring additional benefits for improving the performance of a branch-and-bound algorithm. In comparison, although QNR and QCRE provide better root-node bounds than QCR, they break the sparsity of the problem as the matrix $Z$ introduced in the reformulation is generally not sparse.

In conclusion, based on the results on instances from Biq Mac Library, we can see that QNR+TRI is the most effective method. QNR can be more effective than QCRE on improving the performance of the solver. QCR is efficient on sparse instances with no more than 120 variables.

Finally, we remark that the proposed reformulations are not compared with the tailored algorithms such as Biq Mac, BiqBin and FixingBB directly. However, ignoring the differences caused by the computing environment, we can roughly compare our results reported in Table~\ref{tab1} with the results in Table~6 in \cite{Locatelli2024}. We can see that for the 20 instances with $n=120$, the results of QNR+TRI are comparable with the results of BiqBin and FixingBB. However, for the instances where $n=150$, QNR+TRI is still not comparable with these tailored algorithms. Nonetheless, the proposed reformulation techniques can significantly reduce the gap between the performance of a general-purpose solver and the performance of a tailored solver.

\subsection{Numerical results on randomly generated test instances}\label{sec42}

The results in Section~\ref{sec41} are limited to the instances that have more than 100 variables. These results are unable to fully reflect the advantages of different reformulation methods on instances with various sizes and various densities. In order to compare the four reformulations on instances with a wider range of sizes and different densities, we carry out more experiments using additional test instances.

\begin{table}[!h]
	\caption{Numerical results on generated instances.}\label{tab2}
	{\footnotesize 
	\begin{tabular}{rrrrrrrrr}	
		\toprule
		\multirow{2}*{Instance} & \multicolumn{2}{c}{QCR} & \multicolumn{2}{c}{QCRE}&\multicolumn{2}{c}{QNR}&\multicolumn{2}{c}{QNR+TRI}\\
		\cmidrule{2-9}
		~ &Node & Time &Node & Time &Node & Time &Node & Time \\
		\midrule
		be60.3.1 &1174 &1 &202 &1 &186 &2 &1 &11  \\
		be60.3.2 &778 &0 &37 &1 &5 &13 &1 &21  \\
		be60.3.3 &1132 &0 &47 &1 &20 &7 &1 &27  \\
		be60.3.4 &718 &1 &444 &4 &409 &9 &1 &11  \\
		be60.3.5 &407 &0 &68 &1 &33 &13 &1 &20  \\
		be60.3.6 &1094 &0 &34 &1 &15 &20 &1 &20  \\
		be60.3.7 &1270 &0 &44 &1 &26 &6 &1 &19  \\
		be60.3.8 &1598 &1 &47 &2 &33 &8 &1 &19  \\
		be60.3.9 &572 &0 &54 &1 &3 &12 &1 &16  \\
		be60.3.10 &812 &2 &226 &4 &342 &8 &1 &10  \\
		be60.8.1 &1094559 &1974 &487 &5 &716 &5 &1 &11  \\
		be60.8.2 &767152 &1111 &254 &2 &262 &4 &1 &16  \\
		be60.8.3 &1440 &1 &100 &2 &94 &7 &1 &25  \\
		be60.8.4 &2487035 &4523 &196 &2 &186 &3 &1 &14  \\
		be60.8.5 &561640 &803 &239 &2 &287 &3 &1 &14  \\
		be60.8.6 &1788450 &3270 &177 &2 &186 &1 &1 &25  \\
		be60.8.7 &1068816 &1991 &245 &2 &233 &7 &1 &14  \\
		be60.8.8 &(0.42\%) &7200 &669 &6 &743 &2 &1 &15  \\
		be60.8.9 &176 &0 &1 &1 &1 &1 &1 &22  \\
		be60.8.10 &614234 &1054 &388 &4 &480 &17 &1 &12  \\
		be80.3.1 &721 &3 &652 &8 &709 &17 &1 &20  \\
		be80.3.2 &2055 &1 &83 &3 &36 &8 &1 &55  \\
		be80.3.3 &716 &2 &317 &8 &361 &77 &1 &44  \\
		be80.3.4 &677 &4 &525 &19 &650 &33 &1 &24  \\
		be80.3.5 &700 &2 &229 &6 &315 &21 &1 &25  \\
		be80.3.6 &917 &7 &1101 &19 &1721 &85 &1 &19  \\
		be80.3.7 &671 &2 &655 &9 &678 &20 &1 &22  \\
		be80.3.8 &671 &3 &535 &7 &585 &12 &1 &18  \\
		be80.3.9 &706 &2 &133 &5 &151 &20 &1 &47  \\
		be80.3.10 &694 &3 &763 &9 &1073 &8 &1 &23  \\
		be80.8.1 &(5.12\%) &7200 &1063 &18 &931 &29 &1 &19  \\
		be80.8.2 &(3.33\%) &7200 &277 &6 &318 &22 &1 &18  \\
		be80.8.3 &(6.12\%) &7200 &2329 &42 &2350 &27 &1 &18  \\
		be80.8.4 &(4.92\%) &7200 &542 &9 &657 &22 &1 &27  \\
		be80.8.5 &(5.84\%) &7200 &1005 &31 &966 &14 &1 &17  \\
		be80.8.6 &(12.22\%) &7200 &29046 &290 &38868 &145 &314 &115  \\
		be80.8.7 &1860 &1 &77 &3 &42 &196 &1 &53  \\
		be80.8.8 &(3.47\%) &7200 &184 &5 &166 &15 &1 &24  \\
		be80.8.9 &(9.11\%) &7200 &10018 &116 &13079 &73 &285 &75  \\
		be80.8.10 &(4.11\%) &7200 &217 &7 &324 &23 &1 &25  \\
		be100.3.1 &864 &21 &2623 &112 &2482 &17 &1 &36  \\
		be100.3.2 &726 &9 &928 &29 &942 &20 &1 &37  \\
		be100.3.3 &705 &6 &643 &20 &843 &92 &1 &51  \\
		be100.3.4 &1245 &41 &6006 &168 &8484 &189 &1 &25  \\
		be100.3.5 &666 &4 &498 &22 &559 &77 &1 &53  \\
		be100.3.6 &760 &2 &171 &8 &73 &36 &1 &113  \\
		be100.3.7 &1077 &23 &6179 &145 &7754 &105 &1 &29  \\
		be100.3.8 &741 &4 &160 &7 &173 &71 &1 &106  \\
		be100.3.9 &677 &8 &1786 &68 &1696 &110 &1 &31  \\
		be100.3.10 &887 &15 &4254 &103 &7675 &124 &1 &27  \\
		be100.8.1 &(6.57\%) &7200 &1985 &81 &2540 &74 &1 &40  \\
		be100.8.2 &(5.44\%) &7200 &1296 &54 &1219 &32 &1 &38  \\
		be100.8.3 &(5.87\%) &7200 &3889 &213 &4843 &176 &1 &47  \\
		be100.8.4 &(3.18\%) &7200 &451 &27 &588 &115 &1 &40  \\
		be100.8.5 &(7.07\%) &7200 &7848 &306 &12981 &285 &124 &87  \\
		be100.8.6 &(6.75\%) &7200 &3663 &138 &5501 &127 &24 &70  \\
		be100.8.7 &(13.14\%) &7200 &180525 &5390 &191181 &837 &2240 &576  \\
		be100.8.8 &(8.79\%) &7200 &25825 &1042 &31457 &679 &381 &120  \\
		be100.8.9 &(7.79\%) &7200 &29826 &1172 &33043 &257 &425 &165  \\
		be100.8.10 &(8.01\%) &7200 &32138 &845 &32503 &139 &291 &163  \\
		\bottomrule
	\end{tabular}
}
\end{table}

\begin{table}[h!]
	\centering
	\caption{Root-node relative gaps on generated instances.}\label{tab2b}
	{\footnotesize
	\begin{tabular}{rrrrrrr}	
		\toprule
		\multirow{2}*{Instance} & \multicolumn{2}{c}{SDP} & \multicolumn{2}{c}{SDP+RLT} &\multicolumn{2}{c}{SDP+RLT+TRI}\\
		\cmidrule{2-3} \cmidrule{4-5} \cmidrule{6-7}
		~ &RG & Time &RG & Time &RG & Time \\
		\midrule
		be60.3.1 &6.66\% &0.0 &1.26\% &0.4 &0.00\% &9.8  \\
		be60.3.2 &2.97\% &0.0 &0.01\% &0.5 &0.00\% &19.3  \\
		be60.3.3 &3.26\% &0.0 &0.11\% &0.6 &0.00\% &25.1  \\
		be60.3.4 &11.23\% &0.1 &2.13\% &0.7 &0.00\% &8.8  \\
		be60.3.5 &3.95\% &0.0 &0.29\% &0.6 &0.00\% &18.6  \\
		be60.3.6 &5.20\% &0.0 &0.07\% &0.6 &0.00\% &18.7  \\
		be60.3.7 &4.99\% &0.0 &0.09\% &0.6 &0.00\% &17.5  \\
		be60.3.8 &5.29\% &0.0 &0.28\% &0.7 &0.00\% &17.1  \\
		be60.3.9 &3.81\% &0.0 &0.03\% &0.6 &0.00\% &14.2  \\
		be60.3.10 &8.93\% &0.1 &1.69\% &1.0 &0.00\% &8.3  \\
		be60.8.1 &7.68\% &0.0 &1.98\% &0.8 &0.00\% &7.7  \\
		be60.8.2 &5.90\% &0.0 &1.26\% &0.5 &0.00\% &12.0  \\
		be60.8.3 &5.77\% &0.0 &0.67\% &0.6 &0.00\% &21.3  \\
		be60.8.4 &8.78\% &0.0 &1.56\% &0.7 &0.00\% &9.8  \\
		be60.8.5 &6.45\% &0.0 &1.01\% &0.5 &0.00\% &10.1  \\
		be60.8.6 &6.51\% &0.0 &1.03\% &0.5 &0.00\% &21.6  \\
		be60.8.7 &7.10\% &0.0 &1.18\% &0.5 &0.00\% &9.7  \\
		be60.8.8 &7.72\% &0.1 &2.33\% &0.5 &0.00\% &9.7  \\
		be60.8.9 &2.73\% &0.0 &0.00\% &0.6 &0.00\% &18.9  \\
		be60.8.10 &7.58\% &0.0 &1.68\% &0.5 &0.00\% &8.9  \\
		be80.3.1 &7.82\% &0.0 &1.62\% &1.5 &0.00\% &15.7  \\
		be80.3.2 &3.47\% &0.1 &0.22\% &1.3 &0.00\% &49.9  \\
		be80.3.3 &5.23\% &0.0 &0.71\% &1.8 &0.00\% &40.8  \\
		be80.3.4 &7.76\% &0.1 &1.43\% &2.2 &0.00\% &21.0  \\
		be80.3.5 &6.40\% &0.0 &1.13\% &2.0 &0.00\% &21.5  \\
		be80.3.6 &10.39\% &0.0 &2.52\% &1.8 &0.00\% &14.9  \\
		be80.3.7 &6.10\% &0.0 &1.19\% &2.0 &0.00\% &19.9  \\
		be80.3.8 &7.50\% &0.0 &1.52\% &1.3 &0.00\% &15.3  \\
		be80.3.9 &6.26\% &0.0 &0.68\% &2.0 &0.00\% &43.3  \\
		be80.3.10 &6.21\% &0.1 &1.26\% &1.2 &0.00\% &20.2  \\
		be80.8.1 &6.70\% &0.0 &1.57\% &1.5 &0.00\% &14.1  \\
		be80.8.2 &4.77\% &0.0 &0.73\% &1.0 &0.00\% &15.2  \\
		be80.8.3 &7.73\% &0.0 &2.11\% &1.4 &0.00\% &12.3  \\
		be80.8.4 &6.50\% &0.0 &1.38\% &1.7 &0.00\% &23.2  \\
		be80.8.5 &8.33\% &0.0 &1.80\% &2.9 &0.00\% &13.3  \\
		be80.8.6 &13.79\% &0.0 &5.96\% &1.4 &1.07\% &7.4  \\
		be80.8.7 &3.29\% &0.0 &0.15\% &1.6 &0.00\% &46.1  \\
		be80.8.8 &4.87\% &0.0 &0.62\% &1.1 &0.00\% &18.6  \\
		be80.8.9 &10.45\% &0.0 &4.13\% &1.1 &0.59\% &8.8  \\
		be80.8.10 &5.78\% &0.0 &0.94\% &1.3 &0.00\% &20.3  \\
		be100.3.1 &9.20\% &0.1 &1.81\% &4.7 &0.00\% &30.8  \\
		be100.3.2 &6.22\% &0.1 &1.12\% &3.9 &0.00\% &32.5  \\
		be100.3.3 &5.22\% &0.1 &0.99\% &3.5 &0.00\% &47.4  \\
		be100.3.4 &10.19\% &0.1 &2.92\% &4.5 &0.00\% &20.8  \\
		be100.3.5 &4.63\% &0.1 &0.76\% &4.6 &0.00\% &48.1  \\
		be100.3.6 &3.26\% &0.1 &0.33\% &3.7 &0.00\% &106.7  \\
		be100.3.7 &7.46\% &0.1 &2.48\% &3.4 &0.00\% &23.8  \\
		be100.3.8 &3.80\% &0.2 &0.33\% &3.5 &0.00\% &99.6  \\
		be100.3.9 &7.35\% &0.1 &1.70\% &5.6 &0.00\% &27.9  \\
		be100.3.10 &8.75\% &0.1 &2.25\% &3.9 &0.00\% &23.0  \\
		be100.8.1 &7.59\% &0.0 &1.62\% &3.4 &0.01\% &24.1  \\
		be100.8.2 &6.46\% &0.0 &1.48\% &3.1 &0.00\% &29.3  \\
		be100.8.3 &6.87\% &0.1 &1.79\% &4.9 &0.00\% &25.5  \\
		be100.8.4 &4.16\% &0.1 &0.67\% &4.5 &0.00\% &35.3  \\
		be100.8.5 &7.98\% &0.1 &2.65\% &3.0 &0.27\% &18.9  \\
		be100.8.6 &8.00\% &0.1 &2.27\% &4.1 &0.01\% &21.0  \\
		be100.8.7 &14.55\% &0.1 &5.59\% &5.3 &1.61\% &17.6  \\
		be100.8.8 &10.02\% &0.1 &3.65\% &5.0 &0.64\% &19.8  \\
		be100.8.9 &8.66\% &0.1 &3.12\% &4.3 &0.55\% &19.9  \\
		be100.8.10 &8.76\% &0.1 &3.23\% &2.8 &0.51\% &13.9  \\
		\bottomrule
	\end{tabular}
}
\end{table}		

These additional test instances are not selected from Biq Mac Library, since the remaining test instances are not suitable for systematically analyzing the impact caused by the number of variables and the impact of density. Thus, we generate more test instances with $n\in\{60,80,100\}$ and $d\in\{3,8\}$ (i.e., densities $30\%$ and $80\%$, respectively). The parameters $Q$ and $c$ in \eqref{BQP} are generated based on the distributions described in \cite{Pardalos1990}: The non-diagonal entries of $\frac{1}{2}Q$ are randomly sampled from $\{-50,\ldots,50\}$, the diagonal entries of $\frac{1}{2}Q$ are randomly sampled from $\{-100,\ldots,100\}$, and the entries of $c$ are set to zero. Note that the 50 test instances used in Section~\ref{sec41}, generated by Billionnet and Elloumi in \cite{Billionnet2007}, follows the same distributions described above. We solve all the instances using the four methods. The results are presented in Tables~\ref{tab2}~and~\ref{tab2b}, respectively. The columns in the two tables have the same meanings as those in Tables~\ref{tab1}~and~\ref{tab1b}, respectively.

From the results in  Tables~\ref{tab2}~and~\ref{tab2b}, we can see that QNR+TRI yields zero root-node gaps on most test instances, and solves these instances in the root-node directly. However, for all the cases of low density ($d=3$ and $n\in \{60,80,100\}$), or for the cases of small-size ($n=60$), QNR+TRI  generally needs longer total time than the other three reformulation methods, due to the long computational time for preprocessing. In some cases, the time for solving \eqref{SDP+RLT+TRI} is even longer than the total time of the other methods. However, for the instances that have at least 100 variables and have high density, QNR+TRI is the most competitive one.

For all low-density cases, QCR is the best choice, due to the same reason analyzed in Section~\ref{sec41}: QCR keeps the sparsity of the problem, which may bring additional benefits on accelerating the branch-and-bound algorithm. However, for all high-density cases, QCR performs worse than the other three methods, since the quality of lower bound becomes more important, while the benefit of exploiting the sparsity is marginal.

Finally, we observe that the comparison results between QCRE and QNR are mixed: For the cases where $n=100$ and $d=8$, QNR performs better than QCRE on most instances. On the other hand, for the other cases, QNR is slower than QCRE on most instances. As we have discussed in Remarks~\ref{remark_qcre_qnr}~and~\ref{remark_qnr_qcre_root}, QNR can be more efficient than QCRE when processing nodes in deep layers. However, for small-size problems, most nodes enumerated by the branch-and-bound algorithm are located in shallow layers. On the other hand, the log file of Gurobi indicates that QNR needs longer time than QCRE on handling the root-node. Thus, QNR is slower than QCRE on small-size instances, and becomes more effective on instances with medium to large sizes and with high densities.

In conclusion, for the test instances with up to 100 variables, QCR is the best choice if the density is low. For the cases that have high density, QNR+TRI is the most effective one when $n=100$. For the instances that have high density and few variables ($n\leq 80$), QCRE is the most effective method.

\section{Conclusions}

This paper proposes some new reformulations for the 0-1 quadratic programming problem based on the technique of quadratic nonconvex reformulation. Such a reformulation technique provides more flexibility on adding nonconvex quadratic constraints to the problem formulation, which can tighten the McCormick relaxation bounds of the problem effectively. We further propose a general framework for adding valid inequalities into the problem formulation, and analyze the impact on the reformulation to the McCormick relaxation bound. Based on the proposed framework, a specific relaxation that incorporates the triangle inequalities is proposed.

Our numerical results indicate that one of the proposed reformulations, QNR+TRI, outperforms the other reformulation methods on most instances with 120 or more variables, or with 100 variables and high densities. These results clearly demonstrate the effectiveness of the proposed reformulation technique, which can reduce the root-node gap significantly. On the other hand, for small size problems with 60-100 variables and different densities, there is no unique reformulation method that can achieve the best performance on all different cases. Our numerical results show that QCR is the most effective one on instances with low density and with up to 100 variables. On the other hand, for the instances that have high density and with up to 80 variables, QCRE performs best.

Although Gurobi is still not as efficient as some tailored algorithms such as BiqBin and FixingBB for solving 0-1 quadratic programming problems, the gap between the performance of a general-purpose solver and these tailored algorithms is reduced significantly by using the proposed reformulations. Such improvements bring new insights on how to design effective reformulations of mixed-integer quadratic programming problems to accelerate a solver. We believe that the proposed reformulation techniques can be further extended to more general cases of mixed-integer quadratic programming problems.

Finally, we remark that in addition to just adopting the triangle inequalities, it is possible to further improve QNR+TRI by using more valid quadratic inequalities. We leave this exploration to our future work.

\section*{Declarations}

Lu's research has been supported by the National Natural Science Foundation of China Grant No. 12171151. Deng's research has been supported by the National Natural Science Foundation of China Grant No. T2293774.

All authors certify that they have no affiliations with or involvement in any organization or entity with any financial interest or non-financial interest in the subject matter or materials discussed in this manuscript.

Data availability: Data will be made available on reasonable request.


\bibliography{BinQNR}

\end{document}